\documentclass[11pt]{article}
\usepackage{amsfonts,amsmath}
\usepackage{latexsym}
\usepackage{amsthm}
\usepackage{amscd}
\usepackage{commath}
\usepackage{epsfig}
\usepackage{amssymb}
\usepackage{caption}
\usepackage{enumitem}
\usepackage{mathtools}
\usepackage{booktabs}
\usepackage[table]{xcolor}
\usepackage[toc,page]{appendix}
\usepackage{array}
\newcolumntype{P}[1]{>{\centering\arraybackslash}p{#1}}

\usepackage[english]{babel}

\usepackage[letterpaper,top=2cm,bottom=2cm,left=2.5cm,right=2.5cm,marginparwidth=1.5cm]{geometry}

\usepackage[colorlinks=true, allcolors=blue,bookmarksnumbered]{hyperref}
\hypersetup{
	citecolor = red!80!black
}
\usepackage{color,soul}
\setstcolor{red}
\usepackage{stackengine}
\usepackage{tikz}
\usetikzlibrary{intersections}
\usepackage[export]{adjustbox}
\usepackage{array}

\usepackage[capitalise]{cleveref}
\usepackage{microtype}
\usepackage[percent]{overpic}

\usepackage{lineno}

\usepackage{authblk}

\newtheorem{defn}{Definition}[section]

\newtheorem{theorem}{Theorem}[section]

\newtheorem{thm}{Theorem}[section]

\newtheorem{prop}[defn]{Proposition}
\newtheorem{lem}[defn]{Lemma}

\theoremstyle{remark}
\newtheorem{remark}[defn]{Remark}

\numberwithin{equation}{section}

\newcommand{\bb}{\begin{equation}}
\newcommand{\ee}{\end{equation}}

\newcommand{\origin}{\mathbf{o}}

\newcommand{\F}{\mathcal{F}}
\newcommand{\Fd}{\mathcal{F}^{\bullet}}

\newcommand{\PP}{{\rm \bf P}}
\newcommand{\proba}[1]{\mathbf{P}\left(#1\right)}
 
 \newcommand{\Kn}{\mathbf{K}_{n}} 
  
\newcommand{\lsf}{{\rm \bf \lambda SF}}

\newcommand{\lsfkn}{\lsf(\Kn)}

\newcommand{\shape}{\text{Shape}}
\usepackage{pdfpages}

\newcommand\sut{\,;\;}
\newcommand{\dfn}[1]{\textbf{\textit{#1}}}

\newcommand\ST{\,;\;}

\def\rlabel #1 #2{\begin{equation} \label{#1} #2 \end{equation}}

\def\rproof{\begin{proof}}

\def\Qed{\end{proof}}

\makeatletter 
\tikzset{ 
reuse path/.code={\pgfsyssoftpath@setcurrentpath{#1}} 
} 
\tikzset{even odd clip/.code={\pgfseteorule}, 
protect/.code={ 
\clip[overlay,even odd clip,reuse path=#1] 
(current bounding box.south west) rectangle (current bounding box.north east)
; 
}} 
\makeatother 

\title{\textsc{Local limit of massive spanning forests \\ on the complete graph}}

\author[1]{Matteo \textsc{D'Achille}}
\author[1,2]{Nathana\"el \textsc{Enriquez}}
\author[1]{Paul \textsc{Melotti}}
\affil[1]{Laboratoire de Mathématiques d’Orsay, CNRS, Université Paris-Saclay, 91405, Orsay, France}
\affil[2]{DMA, Ecole Normale Supérieure – PSL, 45 rue d’Ulm, F-75230 Cedex 5 Paris, France}
{\affil[ ]{\texttt {\{matteo.dachille,nathanael.enriquez,paul.melotti\}@universite-paris-saclay.fr}}}

\date{\today}

\renewenvironment{abstract}
 {\small
  \begin{center}
  \bfseries \abstractname\vspace{-.5em}\vspace{0pt}
  \end{center}
  \list{}{%
    \setlength{\leftmargin}{7mm}
    \setlength{\rightmargin}{\leftmargin}%
  }%
  \item\relax}
 {\endlist}

\begin{document}
\maketitle

\begin{abstract}
\noindent
We identify the local limit of massive spanning forests on the complete graph. This generalizes a well-known theorem of Grimmett on the local limit of uniform spanning trees on the complete graph.
\end{abstract}

\bigskip

\section{Introduction and main result}

Let ${G}$ be a graph without self-loops. A \dfn{spanning tree} of $G$ is a connected subgraph of $G$ without cycles and containing all the vertices of $G$. A \dfn{spanning forest} of $G$ is a subgraph of $G$ without cycles and containing all the vertices of $G$.
Denote by $\F({G})$ the set of spanning forests of ${G}$ and for $k\geq 1$ the subset $\F_k({G})$ of spanning forests with $k$ connected components (so that $\F_1({G})$ is the set of spanning trees of $G$). A~\dfn{rooted spanning forest} is a spanning forest with one marked vertex per connected component, also called the~\dfn{root} of that tree. The set of rooted spanning forests is denoted by $\Fd(\mathcal{G})$, and the subset of those made of $k$ trees is denoted by $\Fd_k(\mathcal{G})$. Note that each forest $f \in \F_k(\mathcal{G})$ with connected components $t_1,\dots, t_k$ corresponds to $\prod_{i=1}^k |t_i|$ rooted forests, where $|t|$ stands for the number of vertices in $t$.

\medskip
For $\lambda>0$, we introduce the random \dfn{$\lambda$-massive spanning forest} $\lsf({G})$ whose distribution is defined, for all $f \in \F({G})$ having a set of connected components denoted by $C(f)$, by 
\begin{equation}
\proba{\lsf(G)=f} \propto  \lambda^{|C(f)|} \prod_{t\in C(f)}|t| \; . \tag{1}\label{def.lsf}
\end{equation}

\medskip
\noindent
As $\lambda \downarrow 0$, ${\rm \bf \lambda SF}(G)$ converges in distribution to ${\bf UST}(G)$. Let $\Kn$ be the complete graph with $n$ vertices and with a distinguished vertex $\origin$. Grimmett~\cite[Theorem 3]{grimmett} proved that the local limit of ${\bf UST}(\Kn)$ is a Bienaymé-Galton-Watson tree with reproduction law ${\rm Poisson}(1)$, denoted by ${\rm BGWP}(1)$, conditioned to survive forever. The resulting tree $\mathcal{T}_{0}$ coincides with the random tree obtained by growing a sequence of independent unconditioned ${\rm BGWP}(1)$ on each vertex of the semi-infinite line rooted at $\origin$. The result of Grimmett was recently extended by Nachmias--Peres~\cite[Theorem 1.1]{np22} from ${\bf UST}(\Kn)$ to ${\bf UST}(G)$, where $G$ is any simple, connected, regular graph with divergent degree.

\medskip
On the other hand, massive spanning forests have attracted recent attention, in relation with graph spectra in a series of work by Avena and Gaudillière \cite{AG1,AG2}, or with dimers and limits of near-critical trees in a work by Rey \cite{Rey}. A related model of \emph{cycle rooted spanning forests} introduced by Kenyon \cite{Kbundle} is also investigated on increasing sequences of graph in a recent work of Constantin \cite{Constantin_paper}.

\medskip 
Our aim in this paper is to generalize the result of Grimmett to $\lsf(\Kn)$. In order to state our result, we have to introduce the random tree $\mathcal{T}_{\alpha}$: for $\alpha > 0$, it is obtained by growing a sequence of independent unconditioned ${\rm BGWP}(\frac{1}{1+\alpha})$ on each vertex of a spine rooted at $\origin$ with size given by an independent random variable of law ${\rm Geom}(\frac{\alpha}{1+\alpha})$ (see \cref{fig.1} for a portrait). 

\begin{figure}
\centering
\includegraphics[width=.75\linewidth]{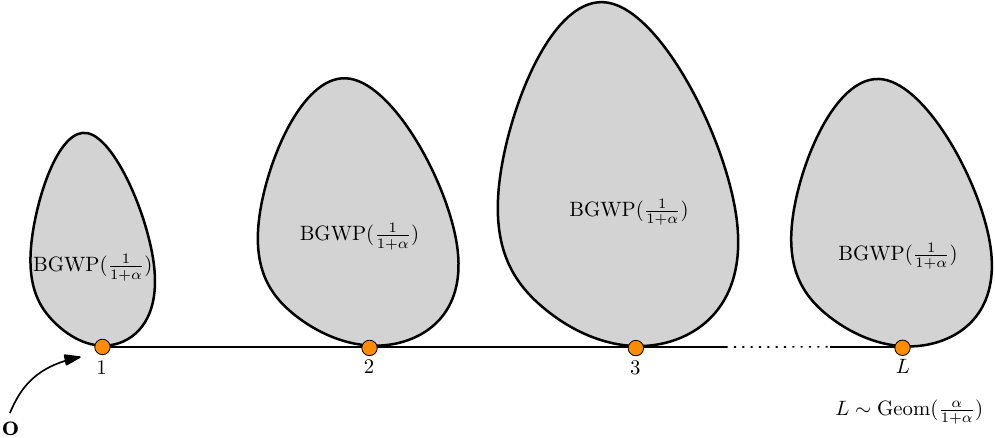}
\caption{The distribution of the local limit $\mathcal{T}_{\alpha}$.}\label{fig.1}
\end{figure}

\medskip
\noindent
Let $T_{n,\lambda}$ be the connected component of $\origin$ in  ${\rm \bf \lambda SF}(\mathbf{K}_{n})$. For a given labeled tree $T$, we denote by $\shape(T,\origin)$ the non-planar unlabeled tree obtained from $T$ and rooted at $\origin$. We endow the set of all locally finite (but possibly infinite) rooted trees with the topology inherited from the product topology. Then the convergence in distribution is called \emph{local convergence}, and coincides with the convergence of probabilities that the tree cut at a finite height is equal to a given pattern. We then have the following local convergence result, for $\lambda_n$ depending on $n$ as $n\to \infty$:

\begin{thm}[\textsc{Local limit of $\lsf(K_{n})$}]\label{thm.lllsf} The following local convergence holds:
$$
\emph{ \shape}(T_{n,\lambda_{n}},\origin) \underset{n \to \infty}{\overset{{\rm (d)}}{\longrightarrow}}
\begin{cases}
\mathcal{T}_{0}  & {\rm if \; } \lambda_{n} = o(n) \\
\mathcal{T}_{\alpha} & {\rm if \; } \lambda_{n} \sim \alpha n \\
\{\origin \} & {\rm if \; } \lambda_{n} \gg n \; .
\end{cases} 
$$
\end{thm}

\begin{remark} 
A corollary of \Cref{thm.lllsf} is that $\mathcal{T}_{\alpha}$ is unimodular.
\end{remark}
\begin{remark}
As shown in \cite[Eq.~19]{CAGM}, the number of connected components has mean $\frac{\lambda+1}{\lambda+n}n$ and is concentrated. When $\lambda_{n}\sim \alpha n$, this gives asymptotically $\frac{\alpha}{1+\alpha}n$. This result can be informally recovered by the computation of the expectation of the inverse of the total progeny of $\mathcal{T}_{\alpha}$, which is equal to $\frac{\alpha}{1+\alpha}$. 
\end{remark}

\medskip

\noindent \textbf{Acknowledgements}. We thank Eleanor Archer and Lucas Rey for insightful discussions. The work of M.A.~was supported by the ERC Consolidator Grant SuperGRandMA (Grant No.~101087572). N.E.~was partially supported by the CNRS grant RT 2179, MAIAGES. M.A.~et N.E.~were partially supported by the CNRS grant RT 2173 Mathématiques et Physique, MP.

\section{Determinantal properties of massive spanning forests on a graph}

Let $H$ be an undirected graph without self-loops but with possibly multiple edges, denote by $V(H)$ and $E(H)$ respectively the vertex set and the edge set of $H$. Let also $n=|V(H)|$ be the \dfn{size} of the graph. Given two vertices $u,v \in V(H)$, we shall denote by $u \sim v$ the fact that there exists an edge connecting $u$ to $v$, and by $u \sim_{e} v$ the fact that a specific edge $e \in E(H)$ connects $u$ to $v$. For a vertex $v \in V(H)$, we shall denote by $d(v)$ the degree of $v$. Given some labeling of the vertices 
$\left(v_{i} \right)_{i=1}^{n}$, the \dfn{adjacency matrix} $A_H$ is defined by

\begin{equation}
\forall i,j=1,\ldots,n \quad \left( A_H\right)_{ij} \overset{\rm def}{=} 
|\lbrace e \in E(H ) \sut  v_i\sim_{e} v_j \rbrace |\,  \, 
\end{equation}

\medskip
\noindent
and the \dfn{graph Laplacian} $\Delta_{H}$ is a symmetric, positive semi-definite matrix defined by

\begin{equation}
\forall i,j=1,\ldots,n \quad (\Delta_{H})_{ij}\overset{\rm def}{=}d(v_{i})\delta_{ij}-\left( A_H\right)_{ij} \, .
\end{equation}
We denote by $P_{H}$ the \dfn{characteristic polynomial} of $\Delta_{H}$ defined by
\begin{equation}\label{eq.carpol}
P_{H}(\lambda)\overset{\rm def}{=}\det (\Delta_{H}+\lambda I_{n}) \; .
\end{equation}

\medskip


This characteristic polynomial turns out be the partition function of spanning forests according to their number of connected components, as shown by the following celebrated Kirchhoff's matrix-forest theorem.
\begin{theorem}[\textsc{Kirchhoff's or Matrix-Forest Theorem \cite{Kir}}]\label{thm.mf}
Let $H$ be a graph of size $n$ without self-loops. Then
\begin{equation}
P_{H}(\lambda) = \sum_{k=1}^n \lambda^k |\Fd_k(H)|
    = \sum_{k=1}^n \lambda^k \left[\sum_{f\in \F_k(H)} \prod_{t\in C(f)}|t|\right].
\end{equation}
\end{theorem}

\noindent
Kenyon generalized this result in~\cite{Kbundle} where he described the partition function of {cycle rooted spanning forests} in terms of {bundle} Laplacian.

\bigskip
A consequence of this fundamental theorem is that the uniform spanning trees, and more generally the $\lsf$ model of spanning forests, are determinantal. This was first discovered by Burton and Pemantle in~\cite{BP} for the uniform spanning tree case, see also the book by Lyons and Peres~\cite[Chapter 4]{LP} for an entire chapter devoted to the topic. As for spanning forests, determinantal formulas for correlations seem to have appeared sporadically. They are described in the PhD thesis of Chang \cite[Section 5.2]{Chang} who considered the equivalent framework of spanning trees rooted at a cemetery point. Concerning the set of roots, determinantal formulas are stated by Avena and Gaudillière in~\cite{AG2}. In this work, we need an expression of presence or absence of edges which is given by a determinant, generalizing~\cite{BP}.

\medskip
For $\lambda \notin {\rm Sp}(\Delta_{H})$, consider the \emph{resolvent} matrix (also called \emph{massive Green's function}) 
\begin{equation}\label{eq.gfdef}
R_\lambda=(\Delta_{H}+\lambda I_n)^{-1}.
\end{equation} Using this operator on $V(H)$, we define another operator on $E(H)$ by first fixing an arbitrary orientation of the edges, so that every edge $e \in E(H)$ has an origin vertex $e_{-}$ and a target vertex $e_{+}$. Then for two edges $e,f\in E(H)$ we define
\begin{equation}\label{eq.kg}
  K_\lambda (e,f) = R_\lambda (e_{-},f_{-}) + R_\lambda (e_{+},f_{+}) - R_\lambda (e_{-},f_{+}) - R_\lambda (e_{+},f_{-}).
\end{equation}
Seen as a matrix, $K$ is called \dfn{transfer current matrix}. This statement says that the $\lsf$ model is determinantal with kernel $K_\lambda$. 
\begin{prop}\label{prop.ag}
  Let $e_1,\dots,e_k, e_{k+1}, \dots, e_p$ be $p$ distinct edges of $H$, let $\lambda>0$. Then
  \begin{equation}
  	\PP(e_1,\dots,e_k \in \lsf(H), \ e_{k+1},\dots,e_p \notin  \lsf(H)) = \det M
  \end{equation}
  where $M$ is the $p\times p$ matrix with entries
  \begin{equation}
  	M_{i,j} = \begin{cases}
  		K_{\lambda}(e_i,e_j) & \text{ if } i\leq k \\
  		\delta_{i,j} - K_{\lambda}(e_i,e_j) & \text{ if } k<i\leq p. 
  	\end{cases}
  \end{equation}
\end{prop}
\begin{proof}
	We start with the case $k=0$, that is, we want to compute $\PP(e_{1},\dots,e_p \notin \lsf(H))$. Let $H\setminus \{e_1,\dots,e_p\}$ be the graph obtained by deleting the edges $\{e_1,\dots,e_p\}$ from $E(H)$. By Theorem~\ref{thm.mf}, this may be expressed as
	\begin{equation}
		   \label{eq.determ}
		    \frac{P_{H \setminus \{e_1,\dots,e_p\}}(\lambda)}{P_{H}(\lambda)} = \frac{\det \left(\Delta_{H \setminus \{e_1,\dots, e_p\}} + \lambda I_n\right)}{\det \left(\Delta_{H} + \lambda I_n\right)}.
	\end{equation}
	The two involved matrices can be related in the following way. Recall that we fixed an arbitrary orientation of every edge. For any $1\leq i \leq p$, let $b_i$ be the vector of dimension $n=|V(H)|$ with entries $-1$ at the origin of $e_i$, $+1$ at the tip of $e_i$, and $0$ otherwise. Then we have
	\begin{equation}
		\Delta_{H \setminus \{e_1,\dots, e_p\}}  = \Delta_{H} - \sum_{i=1}^p b_i b_i^T = \Delta_H - B B^T 
	\end{equation}
	where $B$ is the $n\times p$ matrix with columns $b_i$. Therefore, using the Weinstein–Aronszajn identity,
	\begin{equation}
		\begin{split}
			\det\left(\Delta_{H \setminus \{e_1,\dots, e_p\}}\right) = & \det\left(\Delta_H + \lambda I_n - B B^T \right) \\
			= & \det \left(\Delta_H + \lambda I_n\right) \det\left(I_n - R_\lambda B B^T\right) \\
			= &  \det \left(\Delta_H + \lambda I_n\right)  \det\left(I_p - B^T R_\lambda B \right)
		\end{split}
	\end{equation}
	The entries of $B^T R_\lambda B $ can be seen to be equal to that of $K_\lambda$ by a direct computation. This concludes the proof in the case $k=0$.
	
\medskip
\noindent	
For $k\geq 1$, we can write the probability of the event $\{e_1,\dots,e_k \in \lsf{(H)}, \ e_{k+1},\dots,e_p \notin \lsf{(H)}\}$ using inclusion-exclusion as a combination of probabilities of $\{e_{i_1}, \dots, e_{i_l}, e_{k+1}, \dots, e_p \notin \lsf{(H)}\}$, and those can be expressed via the first case. It is direct to check that multi-linearity of the determinants implies the generic formula.
\end{proof}
 
\clearpage

\section{A formula for the inclusion probability of a tree on $\mathbf{K}_n$}$_{}$

\begin{figure}[!hbtp]
\includegraphics[width=.33\linewidth]{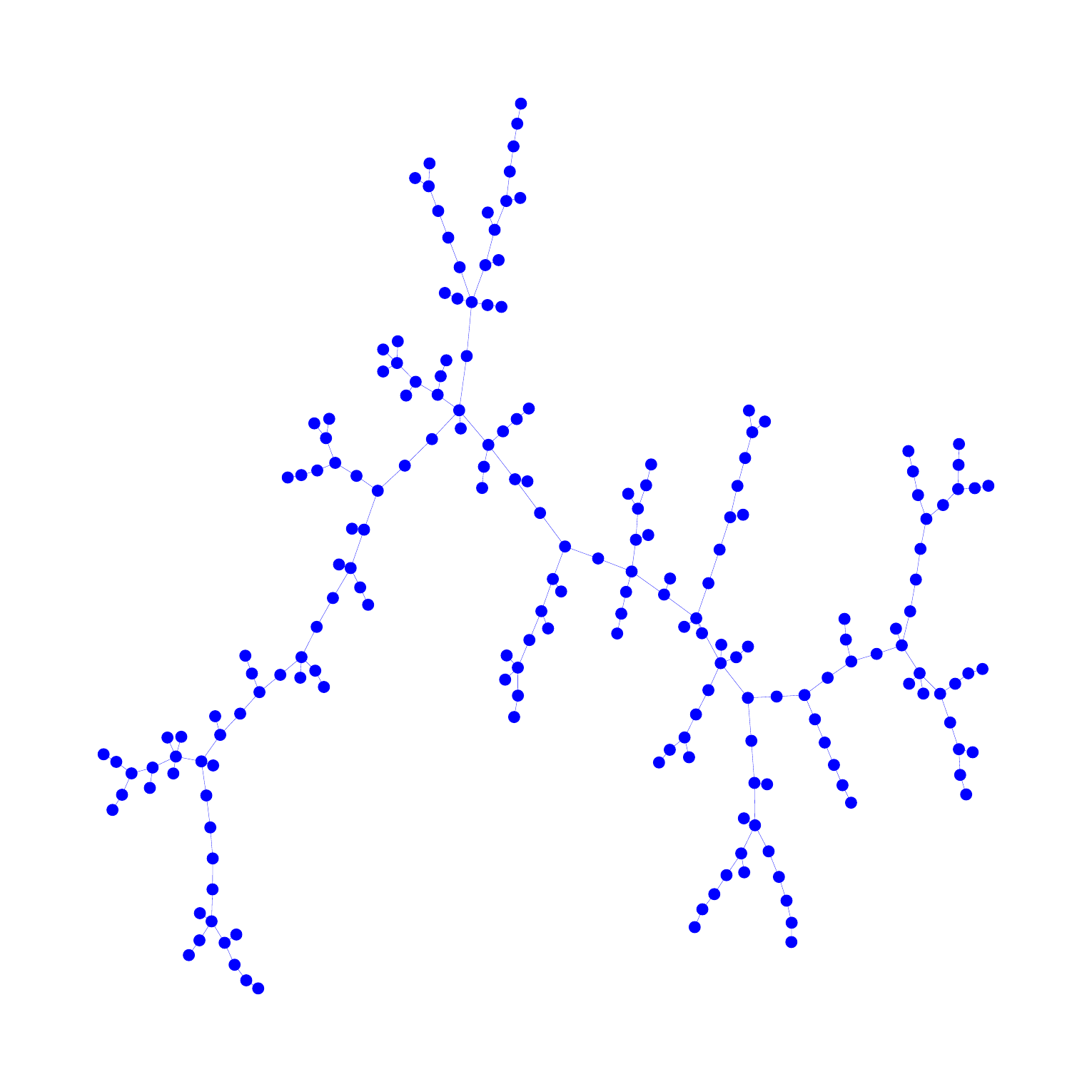}
\includegraphics[width=.33\linewidth]{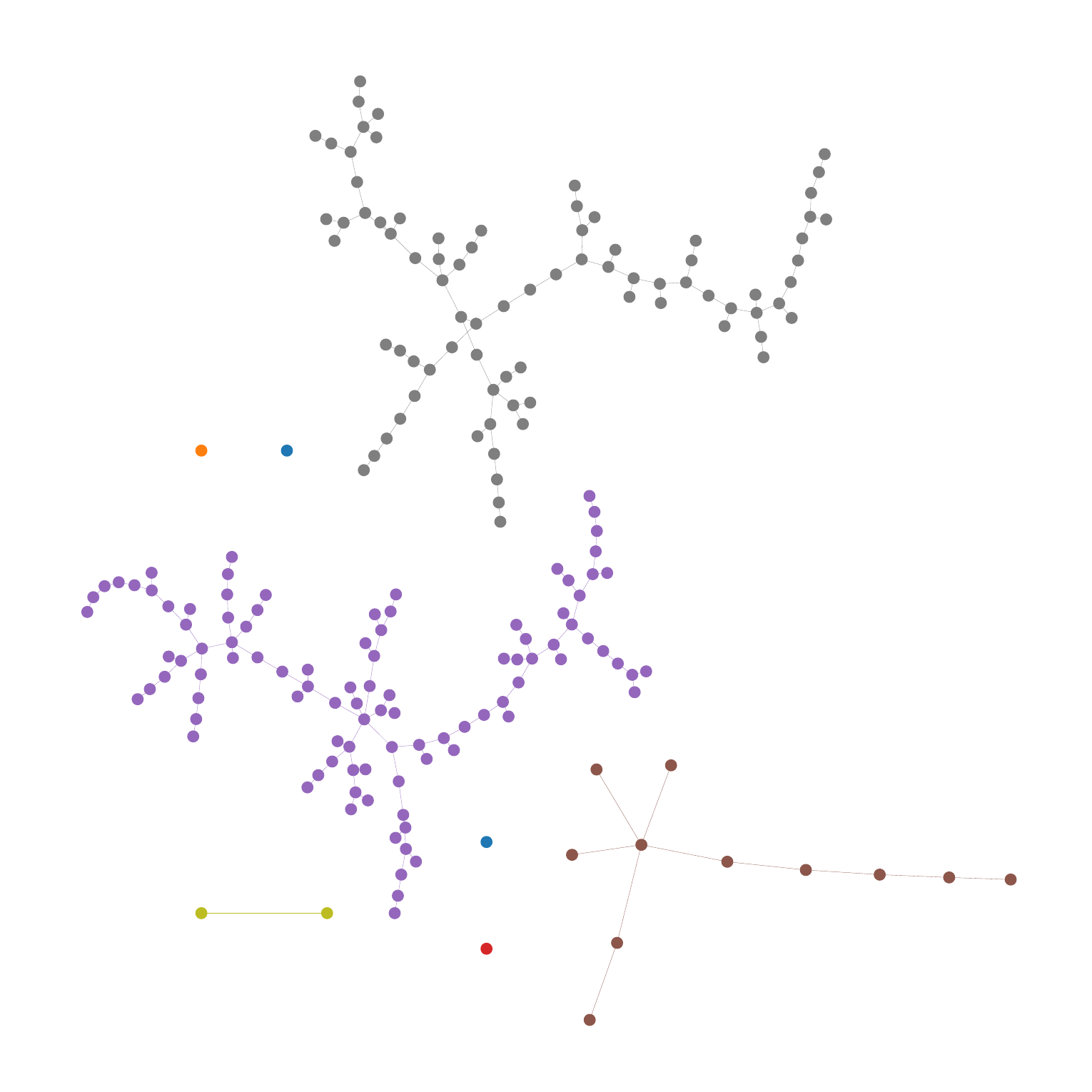}
\includegraphics[width=.33\linewidth]{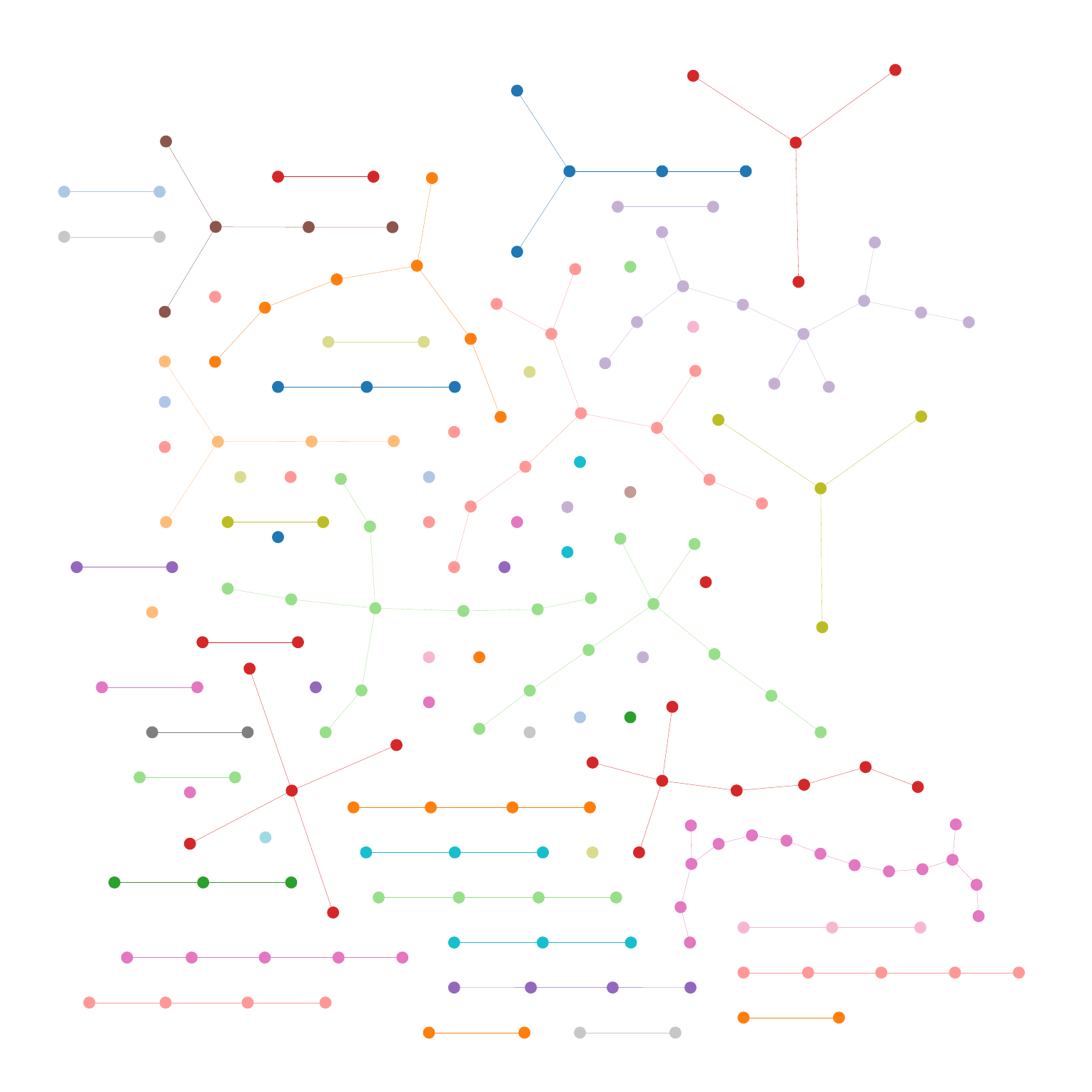}
\caption{Samples of ${\rm \bf \lambda SF}(\mathbf{K}_{200})$ obtained via Wilson's algorithm with a uniform, positive killing rate $r=1-\mu=\frac{\lambda}{\lambda+n-1}$. From left to right: $r=0$ (i.e.~a sample of ${\bf UST}(\mathbf{K}_{200})$), $r=\frac{1}{40}$ and $r=\frac{1}{3}$.}
\label{fig.limvarl}
\end{figure}

\medskip
\noindent
Let $\Kn$ denote the complete graph on $n$ vertices.

\begin{prop}[\textsc{Resolvent and transfer current matrix}] \label{prop.rk}
	The resolvent and transfer current matrix of $\Kn$ are given by
	\begin{equation}
		R_\lambda(i,j) = \begin{cases}
		    \frac{1+\lambda}{\lambda(n+\lambda)} &\text{ if } i=j \\
		   \frac{1}{\lambda(n+\lambda)} &\text{ if } i \neq j
		\end{cases} 
	\end{equation}
and
	\begin{equation}
		K_\lambda(e,f) = 
			\frac{1}{(n+\lambda)}\left(\mathbf{1}_{e_{-}=f_{-}}+\mathbf{1}_{e_{+}=f_{+}}- \mathbf{1}_{e_{-}=f_{+}}-\mathbf{1}_{e_{+}=f_{-}}\right)
	\end{equation}
	
\end{prop}
\rproof

We start with a preparatory computation of the Green's function of a uniform random walk $X^{\mu}_{k}$ on $\Kn$ killed at some positive rate $1-\mu$\footnote{This also provides an easy recipe to sample ${\rm \bf \lambda SF}(\mathbf{K}_{n})$, see \Cref{fig.limvarl}.}, which is equal to $G^{\mu}(x,y)=\sum_{k \geq 0}\mu^{k}\mathbf{P}_{x}(X_{k}=y)$, where $X_{k}$ denotes the uniform RW. Using the notation of \cref{eq.gfdef}, and the fact that the transition matrix of the random walk $P$ is related to the laplacian via $\Delta=(n-1)(I-P)$, we get 
\begin{equation}\label{eq.lgm}
R_{\lambda}(x,y)=\frac{\mu}{n-1}G^{\mu}(x,y) \quad  \mbox{where} \; \mu=\frac{n-1}{n-1+\lambda} \; .
\end{equation}
 By symmetry of the problem, we only need to compute $G^\mu(1,1)$ and $G^\mu(1,2)$. We need now to determine the distribution of $X_{k}$ conditional on $X_{0}=1$. Denoting by $x_{k}\overset{\rm def}{=} \mathbf{P}_{1}(X_{k}=1)$ and $y_{k}\overset{\rm def}{=} \mathbf{P}_{1}(X_{k}=2)$, usual Markov properties show that for $k\geq 0$,
$$
\begin{cases}
	x_{k+1} = y_k, \\
	y_{k+1} = \frac{n-2}{n-1}y_{k}+\frac{1}{n-1}x_k,
	\end{cases}
$$
with $(x_{0}, y_{0})=(1,0)$. From this recursion we deduce that, for all $k\geq 0$, $y_{k+1}-y_{k}=-(-\frac{1}{n-1})^{k+1}$ and using Abel transform, we get $G^\mu(1,2)= \sum_{k\geq 1}\mu^{k} y_{k}=\frac{n-1+\lambda}{\lambda(n+\lambda)}$, giving, via \cref{eq.lgm}, $R_{\lambda}(1,2)=\frac{1}{\lambda(n+\lambda)}$. Since $x_{k+1}=y_{k}$, $G^\mu(1,1)=1+\mu G^\mu(1,2)$, and $R_{\lambda}(1,1)=\frac{\mu}{n-1}+\mu \, R_{\lambda}(1,2)$, which gives the formulas for $R_{\lambda}$.

\noindent
From \cref{eq.kg}, we can write
\begin{equation}
	K_\lambda(e,f) = \begin{cases}
		2 \left(R_\lambda(x,x)-R_\lambda(x,y)\right)\text{ if } e=f \\
		\left(R_\lambda(x,x)-R_\lambda(x,y)\right)\left(\mathbf{1}_{e_{-}=f_{-}}+\mathbf{1}_{e_{+}=f_{+}}- \mathbf{1}_{e_{-}=f_{+}}-\mathbf{1}_{e_{+}=f_{-}}\right) \text{ if } e \neq f
	\end{cases}
	\end{equation}
	
\noindent
which leads to the formulas for $K_{\lambda}$.

\Qed

\medskip
\noindent
We can now compute the probability that a given tree is included in $\lsf(\Kn)$:
\begin{lem}[\textsc{Inclusion probability}]\label{inc.lem}
Let $t$ be a tree with vertices labeled by distinct integers in $\{1,\dots, n\}$, then
$$
\proba{t \subset \lsf(\Kn) }= \frac{|t|}{(n+\lambda)^{|t|-1}} \; .
$$
\end{lem}
\rproof
Let $k=|t|$.
Fix an orientation of $t$.
From \cref{prop.ag}, $\proba{t \subset \lsf(\Kn) }=\det M_{\lambda}$, where $M_{\lambda}$ is matrix of size $(k-1)\times (k-1)$ indexed by the edges of $t$, whose entries are given by Proposition~\ref{prop.rk}. One can check directly that 
$$M_{\lambda}=\frac{1}{n+\lambda} \mathcal{A} \mathcal{A}^{T},$$
where $\mathcal{A}$ is the (rectangular) oriented incidence matrix of $t$: it has size $(k-1)\times k$, with rows indexed by the edges of $t$, columns indexed by the vertices of $t$, and whose nonzero entries are $+1$ (resp. $-1$) when the vertex is the tip (resp. origin) of the edge. Then, by the Cauchy-Binet formula,
$$\det M_\lambda = \frac{1}{(n+\lambda)^{k-1}} \sum_{j=1}^{k} \left(\det \mathcal{A}_{j}\right)^2$$
where $\mathcal{A}_{j}$ is obtained from $\mathcal{A}$ by deletion of column $j$. Then one checks directly that $\det \mathcal{A}_j = \pm 1$, for instance by expanding the determinant over permutations and seeing that only one permutation contributes.
\Qed

\section{Identification of the local limit: Proof of \Cref{thm.lllsf}}

The proof comes in two steps. First, we identify the limit law of the shape of $\lsfkn$ at a given distance $h$ from $\origin$. Second, we show that it coincides with the distribution of $\mathcal{T}_{\alpha}$ restricted at distance $h$ from its root. These steps are the content of the following two propositions.

\bigskip
\noindent
For a given rooted unlabeled tree $T$ and an integer $h$, we denote by:
\begin{itemize}
\item $T_{h}$ the set of vertices of $T$ at distance exactly $h$ from the root;
\item $T_{\leq h}$ the tree given by the intersection of $T$ with the closed ball of radius $h$ centered at the root;
\item $T_{<h}$ the tree given by the intersection of $T$ with the closed ball of radius $h-1$ centered at the root;
\item $\text{Aut}(T)$ the set of graph automorphisms of $T$ preserving its root;
\item $\mathrm{d}(u,v)$ the graph distance between vertices $u$ and $v$.
\end{itemize}
\medskip
\begin{prop} Let $h\in \mathbb{N}$. Let $t$ be an unlabeled tree of height $\leq h$. Then
$$
\proba{\emph{\shape}(\lsfkn,\origin)_{\leq h}=t } = \frac{1}{|\emph{\text{Aut}}(t)|} \frac{(n-1)!}{(n-|t|)! }\frac{1}{(n+\lambda)^{|t|}}\left( n |t_{h}|+\lambda |t| \right) \left(1-\frac{|t_{<h}|}{n+\lambda} \right)^{n-|t|-1} \; .
$$
When $\lambda=\lambda_{n}$,
$$
\lim_{n \to \infty}\proba{\emph{\shape}(\lambda_{n}{\rm \bf SF}(\Kn),\origin)_{\leq h}=t } =
\begin{cases}
 \frac{|t_{h}|}{|\emph{\text{Aut}}(t)|}   {\rm exp}\left(-|t_{<h}| \right)  & {\rm if \; } \lambda_{n} = o(n) \\
\frac{\left(|t_{h}|+\alpha |t|\right)}{|\emph{\text{Aut}}(t)|}  \frac{1}{(1+\alpha)^{|t|}} {\rm exp}\left(-\frac{|t_{<h}|}{1+\alpha} \right)  & {\rm if \; } \lambda_{n} \sim \alpha n \\
\delta_{t=\{ \origin \}} & {\rm if \; } \lambda_{n} \gg n \; .
\end{cases} 
$$
\end{prop}
\rproof

Let $T^{(n)}_{\rm lab}$ be the connected component of $\lsfkn$ containing $\origin$, $T^{(n)}$ its unlabeled version, $t$ as in the statement and $t_{\rm lab}$ a labeled tree whose shape is $t$. Denote the size of $t$ by $|t|(=|t_{\rm lab}|$). We wish to compute the probability of the event $T^{(n)}_{\rm lab, \leq h}=t_{\rm lab}$. This event coincides with the presence of all the edges of $t_{\rm lab}$ in $T^{(n)}_{\rm lab}$ and with the absence of all the edges joining a vertex of $t_{\rm lab,<h}$ to a vertex which does not belong to $t_{\rm lab}$ (of which there are $n-|t|$). Therefore there are $q:=|t_{\rm lab,<h}|\cdot (n-|t|)$ such absent edges, call them $e_{1},\ldots,e_{q}$. By the inclusion-exclusion principle, we can express the probability of this event only in terms of inclusion probabilities and use \Cref{inc.lem}. Namely,
\begin{equation}
\begin{split}
\proba{T^{(n)}_{\rm lab, \leq h}=t_{\rm lab}}&=\proba{t_{\rm lab} \subset T^{(n)}_{\rm lab, \leq h}, e_{1} \notin T^{(n)}_{\rm lab}, \ldots,e_{q} \notin T^{(n)}_{\rm lab}}\\
&=\sum_{k=0}^{q} (-1)^{k}\sum_{\substack{J \subset \{ 1,\ldots,q \}\\ |J|=k}}\proba{t_{\rm lab} \subset T^{(n)}_{\rm lab},\;  \forall j \in J \; e_{j} \in T^{(n)}_{\rm lab} } \; .
\end{split}
\end{equation}
The probability appearing in the second sum on the rhs is equal to $\frac{|t|+k}{(n+\lambda)^{|t|-1+k}}$ if $t_{\rm lab} \cup \{ e_{j} \ST j \in J\}$ is still a tree, and 0 otherwise. A cycle is formed if and only if at least two edges among $\{ e_{j} \ST j \in J\}$ are incident at the same vertex not in $t_{\rm lab}$. Therefore, as soon as $k> n-|t|$, this probability is zero. Otherwise, there are $\binom{n-|t|}{k}$ choices of ``target'' vertices for the tips of $\{ e_{j} \ST j \in J\}$ and the roots of $\{ e_{j} \ST j \in J\}$ can be chosen freely inside $t_{\rm lab,<h}$. We thus get, after straightforward computation,
\begin{equation}\label{eq.lab}
\begin{split}
\proba{T^{(n)}_{\rm lab, \leq h}=t_{\rm lab}} &= \sum_{k=0}^{n-|t|}(-1)^{k}\binom{n-|t|}{k} |t_{<h}|^{k} \frac{|t|+k}{(n+\lambda)^{|t|-1+k}} \\
&=  \frac{1}{(n+\lambda)^{|t|}} \left(n |t_{h}|+\lambda |t|\right) \left(1-\frac{|t_{<h}|}{n+\lambda} \right)^{n-|t|-1}\;.
\end{split}
\end{equation}
Lastly, we want to compute $\proba{{\shape}(\lsfkn,\origin)_{\leq h}=t }$.  For each labeling ${\bf l}$ of $t$, there are exactly $|\text{Aut}(t)|$ labelings ${\bf l'}$ for which the events $\{ T^{(n)}_{{\rm lab}, \leq h}=t_{{\bf l}}\}$ and $\{ T^{(n)}_{{\rm lab}, \leq h}=t_{{\bf l'}}\}$ are the same. On the other hand, there are $\frac{(n-1)!}{(n-|t|)!}$ possible labelings. This concludes the proof for $n$ fixed, and the limits are straightforward.\newline
\Qed
\medskip
\noindent
We turn to the second step, where we study the limiting random tree $\mathcal{T}_{\alpha}$.
We start with a preliminary lemma on standard subcritical $\text{BGWP}(\beta)$.
\begin{lem}\label{lem.gwp}
Let $\beta<1$, let $h\in \mathbb{N}$. Let $t$ be an unlabeled tree of height $\leq h$.

$$
\proba{\emph{\text{BGWP}}(\beta)_{\leq h}=t} = \frac{1}{|\emph{\text{Aut}}(t)|}\beta ^{|t|-1}e^{-\beta |t_{<h}|} \; .
$$ 

\end{lem}
\rproof
We compare the recursion relations satisfied by both quantities. We group the children of the root in equivalence classes: two children are in the same class if and only if they are the root of isomorphic subtrees. We denote now by $\ell$ the number of equivalence classes, and by $n_{1},\ldots, n_{l}$ their cardinality. Since the number of children of the root is $\text{Poisson}(\beta)$ distributed, we can write
\begin{equation}\label{eq.1st}
\begin{split}
\proba{{\text{BGWP}}(\beta)_{\leq h}=t} &= \frac{\beta^{n_{1}+\cdots+n_{l}}e^{-\beta}}{(n_{1}+\cdots+n_{l})!} \binom{n_{1}+\cdots+n_{l}}{n_{1},\ldots,n_{l}} \prod_{i=1}^{l} \proba{{\text{BGWP}}(\beta)_{\leq h-1}=t_{i}}^{n_{i}} \\
& = \frac{\beta^{n_{1}+\cdots+n_{l}}e^{-\beta}}{n_{1}! \cdots n_{l}!} \prod_{i=1}^{l} \proba{{\text{BGWP}}(\beta)_{\leq h-1}=t_{i}}^{n_{i}} 
\end{split}
\end{equation}
\noindent
where $t_{i}$ is the subtree of $t$ emanating from a child of the root belonging to the $i$-th equivalence class. On the other hand, $|{\text{Aut}}(t)|$ satisfies the recursion
$$
|{\text{Aut}}(t)|= \prod_{i=1}^{l} n_{i}! |{\text{Aut}}(t_{i})|^{n_{i}}
$$
hence if we denote by $f$ the function on finite trees which maps $t$ onto $f(t)= \frac{1}{|{\text{Aut}}(t)|}\beta ^{|t|-1}e^{-\beta |t_{<H(t)}|}$, where $H(t)$ is the height of $t$,
\begin{equation}\label{eq.2nd}
\frac{\beta^{n_{1}+\cdots+n_{l}}e^{-\beta}}{n_{1}! \cdots n_{l}!} \prod_{i=1}^{l} f(t_{i})^{n_{i}}=\frac{\beta^{n_{1}+\cdots+n_{l}}e^{-\beta}}{n_{1}! \cdots n_{l}!} \prod_{i=1}^{l} \left( \frac{1}{|{\text{Aut}}(t_{i})|}\beta ^{|t_{i}|-1}e^{-\beta |t_{i,<H(t_{i})}|}\right)^{n_{i}} = f(t) \; .
\end{equation}
\Cref{eq.1st,eq.2nd} show that both quantities of the statement satisfy the same recursion.
\Qed

\medskip
\noindent
We can now identify the distribution of ${\shape}{(\mathcal{T}_{\alpha},\origin)}$:

\begin{prop} Let $h\in \mathbb{N}$. Let $t$ be an unlabeled tree of height $\leq h$. Then
$$
\proba{\emph{\shape}{(\mathcal{T}_{\alpha},\origin)}_{\leq h}=t}= \frac{|t_{h}|+\alpha |t|}{|\emph{\text{Aut}}(t)|}  \frac{1}{(1+\alpha)^{|t|}} {\rm exp}\left(-\frac{|t_{<h}|}{1+\alpha} \right)  \; .
$$
\end{prop}
\rproof

We decompose the probability according to the position of the vertex $v$ of the spine which is the farthest from $\origin$. Let us denote by $[v]$ its equivalence class in the quotient of $t$ by the action of $\text{Aut}(t)$. The vertex $v$ might be at distance $h$ or not; in any case, the probability given that $v$ satisfies our requirement depends only on $[v]$. Calling $L$ the cardinality of the spine, we will use the relation $\proba{L\geq h}=\left(\frac{1}{1+\alpha}\right)^{h}$ when $v \in t_{h}$ and the relation $\proba{L=k+1}=\frac{\alpha}{(1+\alpha)^{k+1}}$ when $\mathrm{d}(v,\origin)=k$. Calling $t_{i,v}$ the tree growing from the $i$-th vertex of the spine, we have
\begin{equation}
\begin{split}
\proba{{\shape}{(\mathcal{T}_{\alpha},\origin)}_{\leq h}=t} &= \sum_{[v]\in t_{h}/\text{Aut}(t)} \left[\prod_{i=0}^{h} \proba{{\text{BGWP}}\left(\frac{1}{1+\alpha}\right)_{h-i+1}=t_{i,v}}\right]\left(\frac{1}{1+\alpha}\right)^{h}+ \\
&+ \sum_{[v] \in t_{<h}/\text{Aut}(t)} \left[\prod_{i=0}^{\mathrm{d}({v},\origin)}\proba{{\text{BGWP}}\left(\frac{1}{1+\alpha}\right)_{h-i+1}=t_{i,v}} \right] \frac{\alpha}{(1+\alpha)^{\mathrm{d}(v,\origin)+1}} \; .
\end{split}
\end{equation}
We use~\Cref{lem.gwp} with $\beta=\frac{1}{1+\alpha}$. Note that 
$$(1+\alpha)^{|t_{1,v}|-1}\cdots (1+\alpha)^{|t_{h,v}|-1} (1+\alpha)^{\mathrm{d}(v,\origin)}  = (1+\alpha)^{|t|},$$
which gives
\begin{equation}
\begin{split}
\proba{{\shape}{(\mathcal{T}_{\alpha},\origin)}_{\leq h}=t} &= \sum_{[v] \in t_{h}/\text{Aut}(t)} \frac{{\rm exp}\left(-\frac{|t_{<h}|}{1+\alpha} \right)}{(1+\alpha)^{|t|}} \prod_{i=0}^{h}\frac{1}{|\text{Aut}(t_{i,v})|}+\\
&+\frac{\alpha}{1+\alpha}\sum_{[v] \in t_{<h}/\text{Aut}(t)}\frac{{\rm exp}\left(-\frac{|t_{<h}|}{1+\alpha} \right)}{(1+\alpha)^{|t|}} \prod_{i=0}^{\mathrm{d}(v,\origin)} \frac{1}{|\text{Aut}(t_{i,v})|} \; .
\end{split}
\end{equation}

\medskip
\noindent
To conclude, we use the Burnside lemma for the action of $\text{Aut}$ on $t_h$ and on $t_{<h}$, giving 
$$|t_{h}| = \displaystyle\sum_{[v] \in t_{h}/\text{Aut}(t)} |\text{Aut}(t)| \prod_{i=0}^{h}\frac{1}{|\text{Aut}(t_{i,v})|} \text{ and } |t_{<h}|=\displaystyle \sum_{[v] \in t_{<h}/\text{Aut}(t)} |\text{Aut}(t)| \prod_{i=0}^{\mathrm{d}({v},\origin)} \frac{1}{|\text{Aut}(t_{i,{v}})|},$$
and the statement follows.

\Qed

\bibliographystyle{alpha}

\bibliography{biblio}

\end{document}